\documentclass[11pt, leqno]{article}

\usepackage{amsmath, amssymb, amscd, amsfonts, enumerate}
\usepackage{amsthm}
\usepackage{latexsym, epsfig, color, verbatim, xcolor, color}
\usepackage{url}

\usepackage{hyperref}
\definecolor{lightblue}{rgb}{0.0,0.4,0.7}
\hypersetup{colorlinks,breaklinks,
  linkcolor=lightblue,urlcolor=lightblue,
anchorcolor=lightblue,citecolor=lightblue}
  
\usepackage[textwidth=50,textsize=tiny]{todonotes}
\newcommand{\C}{\mathbb{C}}

\newcommand{\Z}{\mathbb{Z}}

\newcommand{\Q}{\mathbb{Q}}
\newcommand{\R}{\mathbb{R}}

\newcommand\Disc{\operatorname{Disc}}

\newcommand\bas{\operatorname{bas}}

\newcommand\Res{\operatorname{Res}}

\newcommand\GL{\operatorname{GL}}

\newcommand\ord{\operatorname{ord}}
\newcommand\abs{\operatorname{abs}}

\renewcommand{\Im}{\operatorname{Im}}
\renewcommand{\Re}{\operatorname{Re}}

\newtheorem{proposition}{Proposition}%[section]
\newtheorem{theorem}[proposition]{Theorem}

\newtheorem{lemma}[proposition]{Lemma}
\newtheorem{remark}[proposition]{Remark}

\title{Uniform bounds for lattice point counting and partial sums of zeta functions}

\author{David Lowry-Duda, Takashi Taniguchi, and Frank Thorne}
\begin{document}

\maketitle

%\tableofcontents

\begin{abstract}
We prove uniform versions of two classical results in analytic number theory.

The first is an asymptotic for the number of points of a complete lattice
$\Lambda \subseteq \R^d$ inside the $d$-sphere of radius $R$.
In contrast to previous works, we obtain error terms with implied constants
depending only on $d$.

Secondly, let $\phi(s) = \sum_n a(n) n^{-s}$ be a `well behaved' zeta function.
A classical method of Landau yields asymptotics for the
partial sums $\sum_{n < X} a(n)$, with power saving error terms.
Following an exposition due to Chandrasekharan and Narasimhan, we obtain a version
where the implied constants in the error term will depend only on the `shape of the
functional equation', implying uniform results for families of zeta functions
with the same functional equation.

\end{abstract}

%\tableofcontents

%**********************************************
%**********************************************
%**********************************************
\section{Introduction}\label{intro}
%**********************************************
%**********************************************
%**********************************************

%Our first result is simple to state, and we were surprised to learn that it is apparently not already in the literature.

 Let $\Lambda \subseteq \R^d$ be an arbitrary complete lattice
(i.e., free $\Z$-module of rank $d$),
and consider the counting function
\begin{equation*}
 N(\Lambda, R) := \# \{ v \in \Lambda \ : \ |v| < R\}.
\end{equation*}

We define $r_{\bas}(\Lambda)$ to be the infimum of all $r \in \mathbb{R}^+$ such that
the open ball $B(r)$ of radius $r$ and center $0$ contains a
$\Z$-basis for $\Lambda$.

\begin{theorem}\label{thm:main}
If $R > r_{\bas}(\Lambda)$, then we have
\begin{align}\label{eqn:main}
  N(\Lambda, R)
  =
  \frac{\pi^{d/2}}{\Gamma(d/2 + 1)} \frac{R^d}{\lvert \det(\Lambda) \rvert}
  +
  O_{d}\bigg(
      \frac{1}{\lvert \det \Lambda \rvert}
      r_{\bas}(\Lambda)^{\frac{2d}{d+1}}
    \; R^{d \cdot \frac{d-1}{d+1}}
  \bigg).
\end{align}
\end{theorem}

Note that $r_{\bas}(\Lambda)$ is $O_d(1)$ times the largest successive minimum
of $\Lambda$ (see~\cite[Lemma 8, p. 135]{Cassels}), so that this bound could be phrased in terms
of successive minima instead.

\medskip

Many results like Theorem \ref{thm:main} exist in the literature, and we refer to the comprehensive survey article of
Ivi\'c, Kr\"atzel, K\"uhleitner, and Nowak \cite{lattice_survey} for an overview and numerous references.

We first note that such results may be proved
using the geometry of numbers. One obtains an error term of $O_{d, \Lambda}(R^{d - 1})$:
see Davenport \cite{davenport_lemma} for the basic
principle
%, \cite[Lemma 1]{schmidt} and \cite[Lemma 2]{MV} for results where the
%$\Lambda$-dependence is explicit, 
and Widmer \cite[Theorem 5.4]{widmer} or Ange \cite[Proposition 1.5]{Ange} for versions with a completely explicit error term.

We are interested in the better error terms that come from more analytic techniques. In this context, we could not
find any general result where the dependence of the error term on $\Lambda$ is specified. Such a result
(with a different shape, and a slightly better $R$-dependence of $R^{d - 2}$),
was proved by Bentkus and G\"otze~\cite{bentkus}, but
with the dimension $d$ assumed to be at least $9$.

Our proof is based on classical work of Landau.
It turns out that the Dirichlet series
\begin{equation*}
\zeta(s, \Lambda) := \sum_{v \in \Lambda - \{ 0 \}} |v|^{-2s}
\end{equation*}
are {\itshape Epstein zeta functions}, enjoying analytic continuation and a functional equation of a uniform shape.
Writing $\zeta(s, \Lambda) =: \sum_{n} a(n) \lambda_n^{-s}$, our question is therefore reduced to obtaining error terms
in estimates for the partial sums $\sum_{\lambda_n < X} a(n)$.

This approach was followed in classical work of Landau \cite{Landau1912,
Landau1915}, who obtained \eqref{eqn:main} with the implied constant depending
on $\Lambda$ in an unspecified manner.
Landau, and following him Chandrasekharan and Narasimhan~\cite{ChandrasekharanNarasimhan62}, proceeded by developing general techniques
to bound the partial sums of Dirichlet series with analytic continuation and a functional equation.
Our second main theorem (of which the first will be a consequence) is a uniform
version of this result, valid for a wide class of zeta functions.

We postpone a precise statement to Section~\ref{sec:cn};
the following is a special case.

\begin{theorem}\label{thm:cn}
Let $\phi(s) = \sum_{n} a(n) \lambda_n^{-s}$ be a zeta function with nonnegative coefficients,
absolutely convergent for $\Re(s) > 1$, enjoying an analytic
continuation to $\C$ which is holomorphic away from a simple pole at $s = 1$,
and with a `well behaved' functional equation of degree $d$
relating $\phi(s)$ to $\widehat{\phi}(1 - s)$ for a `dual zeta function' $\widehat{\phi}(s) = \sum_n b(n) \mu_n^{-s}$.

Then, we have
\begin{equation}\label{eq:cn}
\sum_{\lambda_n < X} a(n) = \Res_{s = 1} \big( \phi(s)\big) X + O(X^{\frac{d - 1}{d + 1}}
\delta_1^{\frac{d - 1}{d + 1}} \widehat{\delta_1}^{\frac{2}{d + 1}}),
\end{equation}
provided that the error term is bounded by the main term, and where
\begin{align*}
  \delta_1 &= \Res_{s = 1} \big( \phi(s) \big), \\
 \widehat{\delta_1} &= \sup_Z \frac{1}{Z} \sum_{\mu_n < Z} \lvert b(n) \rvert.
\end{align*}
The implied constant depends on the functional equation, but does not depend further on $\phi(s)$ or the $a(n)$.
\end{theorem}

Here we think of $\delta_1$ as a `density at $s = 1$', and of $\widehat{\delta_1}$ as the `density of the dual',
even if for technical reasons we cannot formulate the latter in terms of a residue, even if the $b(n)$ are nonnegative. We assume above
(as part of being `well behaved') that $\widehat{\delta_1}$ is finite.

We can now describe how to recognize Theorem~\ref{thm:main} as a consequence of
Theorem~\ref{thm:cn}. In terms of the Epstein zeta function $\zeta(s, \Lambda)$,
we recognize that $N(\Lambda, R) = \sum_{\lambda_n \leq R^2} a(n)$.
Applying Theorem~\ref{thm:cn} to $\phi(s) = \zeta(\frac{d}{2} s, \Lambda)$ gives
$N(\Lambda, R^{1/d})$ in terms of
$\delta_1 = \pi^{d/2} \lvert \det \Lambda \rvert^{-1} \Gamma(\frac{d}{2} + 1)^{-1}$
and
$\widehat{\delta_1} = O_d( \lvert \det \Lambda \rvert^{-1} r_{\bas}(\Lambda)^d)$.
Renormalizing to get $N(\Lambda, R)$ gives the statement of Theorem~\ref{thm:main}.
We carry out this investigation in more detail in Section~\ref{sec:proof_main}.

We refer to Section~\ref{sec:cn} for the precise conditions required of the
functional equation in Theorem~\ref{thm:cn}; the definition of
`well behaved' includes (for example) all of the $L$-functions described in \cite[Chapter 5.1]{IK}. Following
\cite{ChandrasekharanNarasimhan62} we stipulate a functional equation
\eqref{eq:functional_base} without any factors of $\pi^{-s/2}$ or involving the `conductor'. These factors should
instead be incorporated into the definition of $\widehat{\phi}(s)$, so that $\mu_n$ will not in general be supported on the integers.
This choice of normalization should be kept in mind when bounding $\widehat{\delta_1}$.
(See Section \ref{sec:proof_ideals} for a typical example.)

Results of a similar flavor were proved by Friedlander and Iwaniec \cite{FI}, by an alternative
classical method. (`Truncating the contour' instead of `finite differencing'.)
In addition, they explain how their results may be further improved when one
can obtain cancellation in certain exponential sums.
(It should be possible, at least in principle, to improve the results of this paper by incorporating
asymptotic estimates for $J$-Bessel functions in place of upper bounds.)

Their method assumes
more of the zeta function; in particular, they assume that its coefficients $a(n)$ are supported on the positive
integers and satisfy the bound $a(n) \ll n^{\epsilon}$.
We are especially interested in
examples, such as Epstein zeta functions, where these hypotheses fail. Some preliminary work suggests that
their method can possibly be made to work without such hypotheses, but that the proofs would not be immediate.

The proof of Theorem~\ref{thm:cn} consists largely of a careful reading
of the analogous proof in~\cite{ChandrasekharanNarasimhan62}. Nevertheless, for the convenience of the reader we present a complete
proof (closely following~\cite[Theorem~4.1]{ChandrasekharanNarasimhan62}). (Our result also eliminates a factor of $X^{\epsilon}$ from
\cite[Theorem~4.1]{ChandrasekharanNarasimhan62}; it was mentioned as
\cite[Remark~5.5]{ChandrasekharanNarasimhan62}, and also seen in
Landau's earlier work, that this was possible.)

%We can apply Theorem \ref{thm:cn} (or its generalization) `out of the box', provided that we can effectively
%bound $\widehat{\delta_1}$ --- often an interesting problem itself.
%Indeed, in the application to Theorem~\ref{thm:main},
%this is where the largest successive minimum $\lambda_n$ makes its appearance.
%this is where the size of the ball containing a basis for $\Lambda$ makes its
%appearance.

\medskip

Another application of `uniform Landau' is the following estimate for the number of ideals
of bounded norm in a number field:
\begin{theorem}\label{thm:ideals}
Let $K$ be a number field of degree $d \geq 1$. Then, the number of integral ideals $\mathfrak{a}$ with
$N(\mathfrak{a}) < X$ satisfies the estimate
\begin{equation}\label{eq:ideals}
\# \{ \mathfrak{a} \ : \ N(\mathfrak{a}) < X \} =
\frac{ 2^{r_1} (2 \pi)^{r_2} h R}{ w |\Disc(K)|^{1/2} } X
+
O\Big( |\Disc(K)|^{ \frac{1}{d + 1}} X^{\frac{d - 1}{d + 1}} (\log X)^{d - 1} \Big),
\end{equation}
if the error term is bounded by the main term, and where the implied constant depends on $d$ only.
\end{theorem}
We prove this theorem for $d \geq 2$ as an application
of Theorem~\ref{thm:cn-full-improved}, our most general version of our main
theorem,
and we remark that for $d=1$ the statement is trivial.
This is very nearly a direct application of Theorem~\ref{thm:cn}, except that
we estimate
$\sum_{\mu_n < Z} |b(n)| \ll Z (\log Z)^{d - 1}$, which amounts to formally taking
$\widehat{\delta_1} = O\big((\log Z)^{d - 1}\big).$
The factor of $(\log X)^{d - 1}$ in \eqref{eq:ideals}
subsumes both this and a related logarithmic factor in $\delta_1$.

We refer to Ange \cite[Corollaire 1.3]{Ange} and Debaene \cite[Corollary 2]{Debaene} for completely explicit bounds, but with
error terms growing more rapidly with $X$.
Moreover, \cite[(66)]{Landau1912} and \cite[(8.20)]{ChandrasekharanNarasimhan62} obtain bounds of essentially the same strength, but with
the implied constant depending on $K$. Following the latter reference, we could also obtain an analogous result with
the additional condition that $\mathfrak{a}$ represent a fixed element of the ideal class group of $K$.

\medskip

There is a further example where Theorem~\ref{thm:cn} is useful: applied to the {\itshape Sato-Shintani
zeta functions}~\cite{SS} associated to a {\itshape prehomogeneous vector space}. This appeared in the work of the second
and third authors~\cite{TT_rc} on counting cubic fields.
The zeta functions in question count cubic {\itshape rings}, and one can also
define zeta functions counting those rings which are `nonmaximal at $q$'. A version of
Theorem~\ref{thm:cn} (appearing implicitly in~\cite{TT_rc}), in combination with a sieve, led to good error terms in the counting function
for cubic fields.
Moreover, these error terms can be further improved --- for this, see~\cite{BTT}, which will apply
essentially the version of Theorem~\ref{thm:cn} stated here, except accounting for secondary poles of the zeta function at
$s = \frac{5}{6}$.

\medskip
Theorem~\ref{thm:main} also has potential applications itself.
The question came to the third author's attention in the course of his work with Kass~\cite{KT}, counting rational
points of bounded height in the Hilbert scheme of two points in the plane.
Some algebraic geometry reduces this to a lattice point counting problem, for which Theorem~\ref{thm:main} applies.
It turns out that a weaker version of Theorem~\ref{thm:main} is equally effective in \cite{KT}, but similar lattice
point counting problems seem likely to arise in related questions counting points on other vector bundles, and Theorem
\ref{thm:main} may prove useful in that context (among others).

\medskip

{\itshape Organization of the paper.} In Section \ref{sec:cn} we state and then prove our most general `uniform
Landau' result (Theorem \ref{thm:cn-full-improved}). We follow Chandrasekharan and Narasimhan \cite{ChandrasekharanNarasimhan62} quite
closely, albeit with a somewhat different exposition, and while removing factors of $X^{\epsilon}$ in the error terms.

We prove Theorem \ref{sec:cn} in Section \ref{sec:simpler}, as a representative (but still fairly general) special case of Theorem \ref{thm:cn-full-improved}. We then prove Theorem
\ref{thm:ideals} in Section \ref{sec:proof_ideals}; once the relevant facts about Dedekind zeta functions are recalled, 
this is also easily deduced from Theorem \ref{thm:cn-full-improved}. 

Finally, we prove Theorem \ref{thm:main} in Section \ref{sec:proof_main}. We must establish a couple of lemmas concerning the geometry of
lattices and their duals, and then the results are again immediate from  Theorem~\ref{thm:cn-full-improved}.

%**********************************************
%**********************************************
%**********************************************
%\section{Introduction}\label{intro}
%**********************************************
%**********************************************
%**********************************************

\section{A uniform version of Landau's method}\label{sec:cn}

We now prove a uniform version of Landau's method, which provides estimates for sums
of coefficients of a Dirichlet series with functional equation. We will closely follow the version given
in~\cite[Theorem~4.1]{ChandrasekharanNarasimhan62}, but indicating the dependence of our estimates on the Dirichlet series itself.
In order to give a complete statement of the theorem, we must set up some notation.

\subsection{Notation and Statement of Theorem}\label{sec:notation}

\begin{itemize}

\item (The Dirichlet series)
Let $\phi(s)$ and $\psi(s)$ denote two dual Dirichlet series,
\begin{equation*}
  \phi(s) = \sum_{n \geq 1} \frac{a(n)}{\lambda_n^s},
  \quad
  \psi(s) = \sum_{n \geq 1} \frac{b(n)}{\mu_n^s},
\end{equation*}
where $\{\lambda_n\}_{n \in \mathbb{N}}$ and $\{\mu_n\}_{n \in \mathbb{N}}$ are two sequences of strictly increasing
positive real numbers tending to $\infty$.
We assume that $\phi(s)$ and $\psi(s)$ each converge absolutely in
a certain fixed half-plane.

\item (The functional equation and meromorphic continuation)
We assume $\phi$ and $\psi$ satisfy a functional equation of the form
\begin{equation}\label{eq:functional_base}
  \Delta(s) \phi(s) = \Delta(\delta - s) \psi(\delta - s),
\end{equation}
where $\delta > 0$ is some real parameter, and
\begin{equation}\label{eqn:def_delta}
  \Delta(s) := \prod_{\nu = 1}^N \Gamma(\alpha_\nu s + \beta_\nu)
  \qquad
  (\alpha_\nu > 0, \beta_\nu \in \mathbb{C})
\end{equation}
is a product of $N \geq 1$ Gamma factors where the $\alpha_\nu$ are positive.
We require $A := \sum_{\nu = 1}^N \alpha_\nu \geq 1$, and note that $2A$ is frequently called the ``degree of the zeta
function.''

We also assume that this functional equation provides meromorphic continuation
in the following sense:
there exists a meromorphic function $\chi$ such that
$\lim_{\lvert t\rvert \to \infty} \chi(\sigma + it) = 0$ uniformly in every
interval $-\infty < \sigma_1 \leq \sigma \leq \sigma_2 < \infty$, satisfying
\begin{align*}
  \chi(s) &= \Delta(s) \phi(s), \qquad \text{for } \Re(s) > c_1, \\
  \chi(s) &= \Delta(\delta - s) \psi(\delta - s), \qquad \text{for } \Re(s) < c_2,
\end{align*}
where $c_1$ and $c_2$ are some constants.

Our hypotheses force all the poles of $\phi(s)$ to be contained within a fixed
vertical strip, and we assume that $\phi(s)$ has only finitely many poles.
This assumption will be necessary for the series in \eqref{eq:def_rphi} to
converge, and so we exclude (for example) Artin $L$-functions (unless the Artin
conjecture is assumed).

\item (Polar Data)
We define
\begin{equation}\label{eq:def_Szero}
  S_\phi^0(X) := \frac{1}{2\pi i} \int_{C_0} \phi(s)
 X^{s} \frac{ds}{s} = \sum_\xi X^{\xi} R_\xi(\log X),
  \end{equation}
  where $C_0$ is any curve enclosing all the singularities of the integrand.
In the latter sum over the poles $\xi$ of $\frac{\phi(s)}{s}$, $R_\xi(\log X)$
is a constant for each simple pole $\xi$, and is generally a polynomial of
degree $\ord_\xi\big(\frac{\phi(s)}{s} \big) - 1$.

We also define
\begin{equation}\label{eq:def_rphi}
R_\phi(X) := \sum_\xi X^{\Re(\xi)} R_\xi^{\abs}(\log X),
\end{equation}
where $R_\xi^{\abs}$ is the polynomial obtained from $R_\xi$ by taking absolute values of each of the
coefficients.

\item (Partial sums)
We denote the partial sum by
\begin{equation*}
  A_\phi^0(X)
  :=
  \sum_{\lambda_n \leq X} a(n).
\end{equation*}

\item (Bounds on partial sums) We require a bound on the partial sums of the coefficients
of the dual zeta function, which we take to be of the form
  \begin{equation}\label{eq:supposition}
   \sum_{\mu_n \leq Z} \lvert b(n) \rvert \leq B_\psi(Z)
  %  \sum_{\mu_n \leq Z} \lvert b(n) \rvert \leq B_\psi Z^{q'} \log^{r'-1} Z
  \end{equation}
for a function $B_\psi(Z)$ of the form
\begin{equation}\label{eq:CZ_bound}
  B_\psi(Z) = C_\psi Z^r \log^{r'}(C_\psi' Z)
\end{equation}
for some $C_\psi,C_\psi' > 0$, $r' \geq 0$ and
$r >  \frac{\delta}{2} + \frac{1}{4A}$.
(We assume $r >  \frac{\delta}{2} + \frac{1}{4A}$ for technical reasons; see \eqref{eq:W_split}.)
For simplicity, we will require this bound simultaneously for all $Z$ for which
the sum in \eqref{eq:supposition} is nonempty, but see Section \ref{subsec:restrict_partial_sum} for a refined version.
\end{itemize}

With these notations, we prove the following theorem.

\begin{theorem}\label{thm:cn-full-improved}
  With the above, we have
  \begin{equation}\label{eq:thm_main_improved}
    A_{\phi}^0(X) - S_\phi^0(X)
    \ll
    X^{- \frac{1}{2A} - \eta} R_\phi(X)
    +
    \sum_{X \leq \lambda_n \leq X + O(y)} \lvert a(n) \rvert
    + X^{\frac{\delta}{2} - \frac{1}{4A}} z^{- \frac{\delta}{2} - \frac{1}{4A}} B_\psi(z),
  \end{equation}
  for every $\eta \geq - \frac{1}{2A}$, and where
  \begin{equation}\label{eq:def_yz}
  y = X^{1 - \frac{1}{2A} - \eta}, \ \ \
  z = X^{2 A \eta} = \frac{X^{2A - 1}}{y^{2A}}.
  \end{equation}
  Moreover, if $a(n) \geq 0$ for all $n$, then the sum over $|a(n)|$ may be omitted, so that we have simply
  \begin{equation}\label{eq:thm_main_improved_positive}
    A_{\phi}^0(X) - S_\phi^0(X)
    \ll
    X^{- \frac{1}{2A} - \eta} R_\phi(X)
    +
     X^{\frac{\delta}{2} - \frac{1}{4A}} z^{- \frac{\delta}{2} - \frac{1}{4A}} B_\psi(z).
  \end{equation}

   Throughout, and in particular in \eqref{eq:thm_main_improved} and
  \eqref{eq:thm_main_improved_positive},
  the implicit constants depend on: the parameter $\eta$, the
  functional equation (i.e.\ on $\delta$, $N$, $\alpha_v$, and
  $\beta_v$), and on the regions in which $\phi$ and $\psi$ converge absolutely -- but not on other data associated to $\phi$ or $\psi$.
\end{theorem}

This is a variation of Theorem~4.1 in~\cite{ChandrasekharanNarasimhan62}, with two
modifications. First of all, we track the dependence of the error terms on growth
estimates for the individual
Dirichlet series $\phi$ and $\psi$.
Secondly, the bound \eqref{eq:supposition} takes the place of
a constant $\beta$ for which
\begin{equation}\label{eq:supposition_CN}
\sum_{n} \lvert b(n) \rvert {\mu_n}^{- \beta} = B'_\psi < \infty,
\end{equation}
avoiding additional factors of $X^{\epsilon}$ appearing in the error terms in \cite{ChandrasekharanNarasimhan62}.
This is not necessarily the only way to do so; indeed, as J. Thorner suggested
to the authors, a plausible alternative approach is to choose $\epsilon =
o_X(1)$ depending explicitly on $X$.

\begin{remark} The bound $\eta \geq - \frac{1}{2A}$ (equivalently, $y \leq X$) is essential;
without it, Landau's finite differencing method doesn't make sense and
counterexamples to the theorem can be constructed.

As is well known, one can at least obtain upper bounds by smoothing; for example, suppose that
the $a(n)$ are nonnegative; then we have
\[
  A_\phi^0(X) \leq  \sum_{\mu_n} a(n) e^{1 - \lambda_n / X}
  = \frac{e}{2 \pi i} \int_{c - i \infty}^{c + i \infty} \phi(s) X^s \Gamma(s) ds.
\]
Now shift the contour to the left of the critical strip, apply the functional equation, and
bound the value of the dual zeta function.
\end{remark}

\subsection{Proof}

We now prove Theorem~\ref{thm:cn-full-improved}.
We defer some proofs of technical lemmas to after the outline to give a better
proof outline.

For each nonnegative integer $k$, we define the smoothed sums
\begin{equation*}
  A_\phi^k(X)
  :=
  \frac{1}{\Gamma(k + 1)} \sum_{\lambda_n \leq X} a(n) (X - \lambda_n)^k.
\end{equation*}
These smoothed sums are sometimes called \emph{Riesz means}.
Typically, it becomes easier to study $A_\phi^k$ for large $k$.
It is possible to recover asymptotics for the non-weighted sum $A_\phi^0(X)$
from asymptotics for $A_\phi^k(X)$ through Landau's ``finite differencing
method.''
Thus the goal is to understand $A_\phi^k(X)$ well.

Recall the notation
\begin{equation*}
  \frac{1}{2\pi i} \int_{(c)} f(s) ds
  :=
  \lim_{T \to \infty} \frac{1}{2\pi i} \int_{-T}^T f(c + it) dt
\end{equation*}
for $c \in \mathbb{R}$.
We recognize $A^k_\phi(X)$ through a classical integral transform (as
in~\cite[\S2]{LowryDudaThesis}, for example) as
\begin{equation}\label{eq:A_base}
  A^k_\phi(X)
  =
  \frac{1}{2\pi i}
  \int_{(c)} \phi(s) \frac{\Gamma(s)}{\Gamma(s + k + 1)} X^{s+k} ds,
\end{equation}
where $c$ is large enough so that the Dirichlet series $\phi(s)$ and $\psi(s)$
converge absolutely for $\Re s \geq c$.
We take $c$ of the form
$c = c(k) = \frac{\delta}{2} + \frac{k}{2A} - \epsilon$
for any $\epsilon$ satisfying $0 < \epsilon < \frac{1}{4A}$,
where the integer $k$ (labeled $\rho$ in \cite{ChandrasekharanNarasimhan62}) is chosen sufficiently large as to guarantee the following properties:
\begin{enumerate}[(i)]
\item
We have $c > -\Re (\beta_\nu /
\alpha_\nu)$ for each $\nu$, guaranteeing that the line $\Re s = c$ is to the
right of all poles of $\Delta(s)/\Delta(\delta - s)$, and that the line $\Re s = \delta - c$
is to the left of all poles of $\phi(s)$.
\item
We have $c > - \Re(\mu / A)$, where $\mu = \frac{1}{2} +
\sum_{\nu = 1}^N (\beta_\nu - \frac{1}{2})$, which we use as a technical
prerequisite to satisfy the conditions of Lemma~\ref{lem:I_rho_X_growth}.
\item\label{it:fractional}
We assume that $\frac{\delta}{2} + \frac{1}{4A} + \frac{k}{2A} > r$ (see \eqref{eq:W_split}), and
that the fractional part of $\frac{k}{2A} - \epsilon - \frac{\delta}{2}$ is in $(0, \frac12)$ (see \eqref{eq:H_shift}).
\item
We assume that $c \neq \delta + n$ for any integer $n$, so that the
integrals \eqref{eq:post_shift} and \eqref{eq:Idef} do not pass through poles.
(In fact, this is implied by \eqref{it:fractional}, since $c-\delta=\frac{k}{2A} - \epsilon - \frac{\delta}{2}$).
\end{enumerate}
As $k$ may be chosen depending only on `the shape of the functional
equation', implied constants in what follows will be allowed to depend on
$k$.

After shifting the line of integration in~\eqref{eq:A_base} to
$\Re s = \delta - c$, replacing $\phi(s)$ with
$\psi(\delta - s) \Delta(\delta - s)/ \Delta(s)$ through the functional
equation~\eqref{eq:functional_base}, and performing the change of variables $s
\mapsto \delta - s$, we rewrite $A_\phi^k(X)$ as
\begin{equation}\label{eq:post_shift}
  A_\phi^k(X)
  =
  S^k_\phi(X)
  +
  \frac{1}{2\pi i} \int_{(c)}
  \frac{\Gamma(\delta - s)}{\Gamma(k + 1 + \delta - s)}
  \frac{\Delta(s)}{\Delta(\delta - s)} \psi(s) X^{\delta + k - s} ds,
\end{equation}
where
\begin{equation}\label{eq:def_Srho}
  S_\phi^k(X) := \frac{1}{2\pi i} \int_{C_k} \phi(s)
  \frac{\Gamma(s)}{\Gamma(s + k + 1)} X^{s + k} ds,
\end{equation}
where $C_k$ is a curve enclosing all the singularities of the integrand between
$\Re(s) = \delta - c$ and $\Res(s) = c$.
(Familiar bounds for the integrand, needed to justify convergence, are recalled in \eqref{eq:gamma_asym}.)

We separate the analytic portion of the shifted integral~\eqref{eq:post_shift}
and define
\begin{equation}\label{eq:Idef}
  I_k(t)
  :=
  \frac{1}{2\pi i}
  \int_{(c)} \frac{\Gamma(\delta - s)}{\Gamma(k + 1 + \delta - s)}
  \frac{\Delta(s)}{\Delta(\delta - s)} t^{\delta + k - s} ds.
\end{equation}
Then we can rewrite~\eqref{eq:post_shift} as
\begin{equation}\label{eq:def_Wp}
  A_\phi^k(X) - S_\phi^k(X)
  =
  W_k(X)
  :=
  \sum_{n \geq 1} \frac{b(n)}{\mu_n^{\delta + k}} I_k(\mu_n X).
\end{equation}
In order to study $W_k(X)$, we will need the following properties of
$I_k(X)$.

\begin{lemma}\label{lem:I_rho_X_growth}
  Suppose that $k$ is large enough that the line $\Re s = c(k)$ is to the
  right of all poles of $\Delta(s)/\Delta(\delta - s)$.
  Let $I_k^{(k)}$ denote the $k$th derivative of $I_k$.
  Then for $t \geq 1$, we have
  \begin{equation*}
    I_k(t) \ll t^{\frac{\delta}{2} - \frac{1}{4A} + k(1 - \frac{1}{2A})},
    \qquad
    I_k^{(k)}(t) \ll t^{\frac{\delta}{2} - \frac{1}{4A}}.
  \end{equation*}
  As $t \to 0$, we have that
  \begin{equation*}
    I_k(t) \ll t^{\frac{\delta}{2} + k(1 - \frac{1}{2A}) + \epsilon},
    \qquad
    I_k^{(k)}(t) \ll t^{\frac{\delta}{2} + \epsilon}.
  \end{equation*}
\end{lemma}

\begin{proof}
  Proved in Section~\ref{sec:lemproofs}. 
\end{proof}

We are now ready to describe the finite differencing method, which we apply
to~\eqref{eq:def_Wp}.
Define $\Delta_y F(X) := F(X + y) - F(X)$,
so that the $k$th finite difference operator $\Delta_y^k$ is given by
\begin{equation}\label{eq:def_finite_difference}
  \Delta_y^k F(X)
  =
  \sum_{\nu = 0}^k (-1)^{k - \nu} {k \choose \nu} F(X + \nu y).
\end{equation}
(See \eqref{eq:finite_diff_iterated} for an alternative formula when $F$ is $k$ times differentiable.)

\begin{lemma}\label{lem:finite_diffs_make_dreams_come_true}
  With the same notation as above,
  \begin{equation*}
    \Delta_y^k A_\phi^k(X)
    =
    A_\phi^0(X) y^k
    +
    O \Big( y^k \sum_{X \leq \lambda_n \leq X + k y} \lvert a(n) \rvert\Big).
  \end{equation*}
  Additionally, recalling the definitions of $R_\phi$ and $S_\phi^k(X)$ from
  \eqref{eq:def_rphi} and \eqref{eq:def_Srho} respectively,
 we have for $y \ll X$ that
  \begin{equation}\label{eq:S_diff}
    \Delta_y^k S_\phi^k(X)
    =
    S_\phi^0(X) y^k
    +
    O\Big( \frac{y^{k + 1}}{X} R_\phi(X) \Big).
  \end{equation}
\end{lemma}

\begin{proof}
  Proved in Section~\ref{sec:lemproofs}.
  %(Note that, if $\phi$ is entire, \eqref{eq:S_diff} asserts that $0 = 0$.)
\end{proof}

We apply $\Delta_y^k$ to \eqref{eq:def_Wp}.
For the left hand side of \eqref{eq:def_Wp}, we see from above that
\begin{equation}\label{eq:AS_finite_diff_setup}
  \Delta_y^k [ A_\phi^k(X) - S_\phi^k(X)]
  =
  y^k [A_\phi^0(X) - S_\phi^0(X)]
+
  O \Big(
    \frac{y^{k + 1}}{X} R_{\phi}(X)
    +
    y^k \sum_{X \leq \lambda_n \leq X + k y} \lvert a(n) \rvert
  \Big).
  \end{equation}
On the other side of~\eqref{eq:def_Wp}, we get
\begin{equation}\label{eq:W_finite_diff_setup}
  \Delta_y^k W_k(X)
  =
  \sum_{n \geq 1} \frac{b(n)}{\mu_n^{\delta + k}} \Delta_y^k I_k(\mu_n X).
\end{equation}
Note that the finite difference is taken of $I_k(\mu_n X)$ as a function of $X$, not of $\mu_n X$.
Using the properties of $I_k(X)$ as stated in Lemma~\ref{lem:I_rho_X_growth},
one can prove the following lemma.

\begin{lemma}\label{lem:I_finite_diffs}
  For $y \ll X$, we have
  \begin{equation}\label{eq:I_finite_diffs}
    \Delta_y^k I_k (\mu_n X) \ll \begin{cases}
      \max_{t \asymp \mu_n X} \lvert I_k (t) \rvert
      \ll
      (\mu_n X)^{\frac{\delta}{2} - \frac{1}{4A} + k(1 - \frac{1}{2A})},
      \\
      (\mu_n y)^k \max_{t \asymp \mu_n X} \lvert I_k^{(k)}(t) \rvert
      \ll
      (\mu_n y)^k (\mu_n X)^{\frac{\delta}{2} - \frac{1}{4A}}.
    \end{cases}
  \end{equation}
\end{lemma}

\begin{proof}
  Proved in Section~\ref{sec:lemproofs}.
\end{proof}

The first bound in~\eqref{eq:I_finite_diffs} is superior to the second bound
when $\mu_n \gg z := X^{2A - 1}/y^{2A}$, so that we get the bound
\begin{align}\label{eq:W_split}
%\begin{split}
  \Delta_y^k W_k(X)
  \ll & \
  y^k X^{\frac{\delta}{2} - \frac{1}{4A}}
  \sum_{\mu_n \leq z} |b(n)| \mu_n^{-\frac{\delta}{2} - \frac{1}{4A}}
  +
  X^{\frac{\delta}{2} - \frac{1}{4A} + k(1 - \frac{1}{2A})}
  \sum_{\mu_n > z}
  |b(n)| \mu_n^{- \frac{\delta}{2} - \frac{1}{4A} - \frac{k}{2A}} \\
  \ll & \
   y^k X^{\frac{\delta}{2} - \frac{1}{4A}} z^{- \frac{\delta}{2} - \frac{1}{4A}} B_\psi(z)
   +
   X^{\frac{\delta}{2} - \frac{1}{4A} + k(1 - \frac{1}{2A})}
    z^{- \frac{\delta}{2} - \frac{1}{4A} - \frac{k}{2A}}
    B_\psi(z), \label{eq:Wp_semifinal}
   \end{align}
where in the latter step we deviated
from \cite{ChandrasekharanNarasimhan62}
by dividing the sums into dyadic intervals $[\frac{Z}{2}, Z]$,
bounding the contribution of each by \eqref{eq:supposition}, and using \eqref{eq:CZ_bound}
to sum the results. Our choice of $z$ equalizes the two terms in \eqref{eq:Wp_semifinal}, so that
the second of them may be omitted.

Therefore applying finite difference operators to~\eqref{eq:def_Wp} and
inserting the bounds for the left hand side~\eqref{eq:AS_finite_diff_setup},
and the right hand side~\eqref{eq:Wp_semifinal}, we see that
\begin{equation}\label{eq:combined_bound_I}
  A_\phi^0(X) - S_\phi^0(X)
  \ll
    \frac{y}{X} R_\phi(X)
    +
    \sum_{X \leq \lambda_n \leq X + k y} \lvert a(n) \rvert
  +
  X^{\frac{\delta}{2} - \frac{1}{4A}} z^{- \frac{\delta}{2} - \frac{1}{4A}} B_\psi(z),
\end{equation}
which is \eqref{eq:thm_main_improved}, after the change of variables
$y = X^{1 - \frac{1}{2A} - \eta}$ for some $\eta \geq -\frac{1}{2A}$.

Suppose further now that $a(n) \geq 0$ for all $n$.
Then, as noted in~\cite[eq.\ 4.15]{ChandrasekharanNarasimhan62}, $A_\phi^0(X)$ is monotone in $X$ and we have that
\begin{equation}\label{eq:monotone_inequality}
  y^k A_\phi^0(X)
  \leq
  \Delta_y^k A_\phi^k (X)
  \leq
  y^k A_\phi^0(X + k y).
\end{equation}

This may be proved using~\eqref{eq:finite_diff_iterated} on $A_\phi^k(X)$.
For $i \geq 1$, it's true that $\frac{d}{dX}A_\phi^{i+1}(X)=A_\phi^i(X)$, but
one must check that~\eqref{eq:finite_diff_iterated} is true of $A_\phi^k(X)$
even though $A_\phi^1(X)$ is not differentiable when $X$ is an integer.

Using the inequalities~\eqref{eq:monotone_inequality} with~\eqref{eq:S_diff} gives that
\begin{align*}
  A_\phi^0(X) - S_\phi^0(X)
  \leq
  y^{-k} \Delta_y^k W_k(X)
  +
  O\Big( \frac{y}{X} R_\phi(X) \Big),
\end{align*}
and  estimating $ \Delta_y^k W_k(X)$ as before we obtain
\eqref{eq:thm_main_improved_positive} as an upper bound for $A_\phi^0(X) - S_\phi^0(X)$, and similarly as a lower
bound for $A_\phi^0(X + k y) - S_\phi^0(X)$.
Since $S_\phi^0(X + k y) - S_k^0(X) \ll \frac{y}{X} R_\phi(X)$,
we obtain \eqref{eq:thm_main_improved_positive} as a lower bound for $A_\phi^0(X + k y) - S_\phi^0(X + k y)$,
and correspondingly for $A_\phi^0(X) - S_\phi^0(X)$ after a suitable change of
variables.

This completes the proof of Theorem~\ref{thm:cn-full-improved}.

\subsection{Proofs of Technical Lemmas}\label{sec:lemproofs}

\begin{proof}[Proof of Lemma~\ref{lem:I_rho_X_growth}]

Define
\begin{equation*}
  G(s) := \frac{\Gamma(\delta - s)}{\Gamma(k + 1 + \delta - s)}
  \frac{\Delta(s)}{\Delta(\delta - s)},
\end{equation*}
so that $I_k$ is an inverse Mellin transform of $G(s)$.
We will show that $G(s)$ can be compared to a function $H(s)$, whose inverse
Mellin transform can be explicitly evaluated in terms of $J$-Bessel functions.
As a consequence of Stirling's approximation, one can
show~\cite[2.12]{ChandrasekharanNarasimhan62} that for any $\alpha$,
\begin{equation*}
  \log \Gamma(z + \alpha)
  =
  (z + \alpha - \tfrac{1}{2}) \log z
  - z +
  \tfrac{1}{2} \log 2\pi + O(\lvert z \rvert^{-1})
\end{equation*}
as $\lvert z \rvert \to \infty$, uniformly in regions
$\lvert \arg z \rvert < \pi - \delta$ for any fixed $\delta > 0$.
Using this expression on $G(s)$, one can show that
\begin{equation}\label{eq:gamma_asym}
G(s) \asymp |\Im(s)|^{2A\sigma - A\delta - (k + 1)}
\end{equation}
uniformly on any fixed vertical strip,
and further that
\begin{equation}\label{eq:lem1_1}
  \log G(s)
  -
  \log \Big(
    \frac{\Gamma(As + \mu)}{\Gamma(\lambda - As)}
    e^{\Theta s}
  \Big)
  =
  B + O(\lvert s \rvert^{-1}),
\end{equation}
where
\begin{align*}
  \mu &= \frac{1}{2} + \sum_{\nu = 1}^N (\beta_\nu - \frac{1}{2}) \\
  \lambda &= \mu + A\delta + k + 1, \\
  \Theta &= 2 \Big( \sum_{\nu = 1}^N \alpha_\nu \log \alpha_\nu
  - A \log A \Big), \\
  B &= -\delta \sum_{\nu = 1}^N \alpha_\nu \log \alpha_\nu
  + (A \delta + k + 1) \log A.
\end{align*}
We therefore have
\begin{equation}\label{eq:lem1_Irho}
  I_k(t)
  =
  \frac{1}{2\pi i} \int_{(c)} H(s) t^{\delta + k - s} ds
  +
  \frac{1}{2\pi i} \int_{(c)}
  \big(G(s) - H(s)\big) t^{\delta + k - s} ds,
\end{equation}
where we define $H(s)$ to be
\begin{equation}\label{eq:Hdef}
  H(s) = \frac{\Gamma(As + \mu)}{\Gamma(\lambda - As)} e^{B+\Theta s},
\end{equation}
and we note that it follows from~\eqref{eq:lem1_1} that
\begin{equation*}
  G(s) - H(s) = H(s) \cdot O(\lvert s \rvert^{-1}).
\end{equation*}

Suppose first that $t \geq 1$.
For the second term in~\eqref{eq:lem1_Irho}, we shift the line of integration
to $\Re s = c + \frac{1}{2A}$.
Our assumption \eqref{it:fractional} on $k$ imply that we do not pass through any poles, and the
shifted integral converges absolutely by \eqref{eq:gamma_asym}, so that
\begin{align}\label{eq:H_shift}
  \frac{1}{2\pi i} \int_{(c)} \big(G(s) - H(s)\big) t^{\delta + k - s} ds
  & =  \frac{1}{2\pi i} \int_{(c + \frac{1}{2A})} H(s)
  \cdot O(\lvert s \rvert^{-1}) t^{\delta + k - s} ds \\
&  \ll
  t^{\delta + k - c - (1/2A)} \nonumber \\
 &  \ll
  t^{\frac{\delta}{2} + \frac{2A-1}{2A} \cdot k - \frac{1}{4A}}. \nonumber
\end{align}

For the first term in~\eqref{eq:lem1_Irho}, we recognize it as a $J$-Bessel
function~\cite{WatsonBessel}
\begin{equation}\label{eq:JBessel}
  \frac{1}{2\pi i} \int_{(c)} H(s) t^{\delta + k - s} ds
  =
  A_1 \, (\tilde{t}^{\, 1/2A})^{A\delta + (2A-1)k}
  J_{2\mu +A\delta +k}(2\tilde{t}^{\, 1/2A})
\end{equation}
for a positive constant $A_1$ and where
$\tilde{t} = te^{-\Theta}$ is a linear change of variables.
Using the classical bound $J_\nu(x) \ll x^{-1/2}$ (as
in~\cite[(2.12)]{ChandrasekharanNarasimhan62}
or~\cite{WatsonBessel}), we see that~\eqref{eq:JBessel}, and hence also \eqref{eq:Idef}, is bounded by
\begin{equation*}
  \ll (t^{1/2A})^{A\delta + (2A-1)k} t^{-1/4A}
  =
  t^{\frac{\delta}{2} + \frac{2A-1}{2A} \cdot k - \frac{1}{4A}}.
\end{equation*}

As $t \to 0$, the bound
$I_k(t) \ll t^{\frac{\delta}{2} + \frac{2A-1}{2A}k + \epsilon}$
follows from immediately bounding the integrand in~\eqref{eq:Idef} absolutely.

These prove the two bounds for $I_k(t)$.
We now prove the corresponding bounds for $I_{k}^{(k)}(t)$.
The argument is largely the same as above.
With $c_0 = \frac{\delta}{2} - \epsilon$ (which is $c_k$ when $k = 0$),
define a contour $C'$ as follows:
from $c_0 - i \infty$ up to $c_0 - iR$, right to $c_0 + r - iR$,
up to $c_0 + r + iR$, left to $c_0 + iR$, up to $c_0 + i\infty$.
The parameters $r$ and $R$ are chosen as large as necessary so that passing the contour from the line
$\Re s = c_{k}$ to $C'$ does not cross any poles.

Thus shifting the contour, and differentiating under the integral sign, we have
\begin{equation}\label{eq:Irhorho_1}
  I_k^{(k)}(t)
  =
  \frac{1}{2 \pi i} \int_{C'} h(s) t^{\delta - s} ds
  +
  \frac{1}{2 \pi i} \int_{C'} (g(s) - h(s)) t^{\delta - s} ds,
\end{equation}
where
\begin{equation*}
  g(s) = \frac{\Delta(s)} {(\delta - s) \Delta(\delta - s)},
\end{equation*}
and $h(s)$ is defined as in $H(s)$ (in~\eqref{eq:Hdef}), but with $k = 0$ in
the parameter $\lambda$.
As before
\begin{equation*}
  g(s) - h(s) = h(s) \cdot O(|s|^{-1}).
\end{equation*}

The second integral is bounded analogously to the integral
of $G(s) - H(s)$ above, by shifting to the right, giving for $t \to 0$
\begin{align*}
  \frac{1}{2 \pi i} \int_{C'} (g(s) - h(s)) t^{\delta - s} ds
  &=
  \frac{1}{2 \pi i} \int_{C' + \frac{1}{2A}} h(s)
  \cdot O(\lvert s \rvert^{-1}) t^{\delta - s} ds
  \\
  &\ll t^{\frac{\delta}{2} - \frac{1}{2A} + \epsilon}
  \ll t^{\frac{\delta}{2} - \frac{1}{4A}}.
\end{align*}
The first integral can similarly be explicitly evaluated in terms
of a the $J$-Bessel function.
Elementary manipulations as above show
\begin{equation*}
  \frac{1}{2 \pi i} \int_{C'} h(s) t^{\delta + s} ds
  =
  A_1 \tilde{t}^{\delta/2} J_{2 \mu + A \delta}(2 \tilde{t}^{1/2A})
  \ll t^{\delta/2 - \frac{1}{4A}}.
\end{equation*}
Finally, we have $I_k^{(k)}(t) \ll t^{\frac{\delta}{2} + \epsilon}$ as $t \to 0$ by
trivially bounding~\eqref{eq:Irhorho_1} on the initial line of integration.
This completes the proof.
\end{proof}

\begin{proof}[Proof of Lemma~\ref{lem:finite_diffs_make_dreams_come_true}]

Applying the finite differencing operator $\Delta_y^k$ directly to
$A_\phi^k(X)$ gives that
\begin{align*}
  \Delta_y^{k} A_{\phi}^{k}(X)
  &=
  \sum_{\lambda_n \leq X}
  a(n) \frac{\Delta_y^{k} (X - \lambda_n)^{k}}{\Gamma(k + 1)}
  \\
  &\quad +
  \frac{1}{\Gamma(k + 1)}
  \sum_{\nu = 0}^{k}
  (-1)^{k - \nu} {k \choose \nu}
  \sum_{\lambda_n \in (X, X + \nu y]} a(n) (X + \nu y - \lambda_n)^{k}
  \\
  &=
  A^0_\phi(X) y^{k} +
  O\big(
    y^{k} \sum_{X \leq \lambda_n \leq X + k y} \lvert a(n)\rvert
  \big).
\end{align*}
We have used the explicit evaluation
$\Delta_y^k(X - \lambda_n)^k = y^k \Gamma(k+1)$
to simplify this expression;
for a $k$-times differentiable function $F$, one can use induction on
$k$ to show that
\begin{equation}\label{eq:finite_diff_iterated}
  \Delta_y^k F(x)
  =
  \int_x^{x+y} dt_1 \int_{t_1}^{t_1 + y} dt_2
  \cdots
  \int_{t_{k - 1}}^{t_{k - 1} + y} F^{(k)}(t_k) dt_{k}.
\end{equation}
We also use \eqref{eq:finite_diff_iterated} to prove \eqref{eq:S_diff}:
Since the $k$th derivative of $S_\phi^k(X)$ is exactly $S_\phi^0(X)$ (for any $c$ satisfying the
listed hypotheses), we
then have that
\begin{equation}\label{eq:S_iterated_integral}
  \Delta_y^k S_\phi^k(X)
  =
  \int_X^{X+y} dt_1 \int_{t_1}^{t_1 + y} dt_2
  \cdots
  \int_{t_{k - 1}}^{t_{k - 1} + y} S_\phi^0(t_k) dt_{k}.
\end{equation}
The result then follows by writing $S_\phi^0(t)$ in terms of the residues of $\phi(s)$,
as in \eqref{eq:def_Szero}, and substituting into ~\eqref{eq:S_iterated_integral}.

\end{proof}

\begin{proof}[Proof of Lemma~\ref{lem:I_finite_diffs}]

For $y \ll X$, we have the trivial inequality using only the definition of
the finite differencing operator $\Delta_y^k$,
\begin{equation*}
  \Delta_y^k I_k(\mu_n X)
  \ll
  \max_{t \asymp \mu_n X}
  \lvert I_k(t)\rvert.
\end{equation*}

For the second bound, we use~\eqref{eq:finite_diff_iterated} to see that
\begin{equation*}
  \Delta_y^k I(\mu_n X)
  =
  \int_{X}^{X + y} dt_1 \int_{t_1}^{t_1 + y} dt_2
  \cdots
  \int_{t_{k - 1}}^{t_{k - 1} + y} I^{(k)}_k(\mu_n t_k) d t_k
  \ll
  y^k \max_{t \asymp \mu_n X} I_k^{(k)}(t),
\end{equation*}
where we have trivially bounded the iterated integrals in the last inequality.

In both cases the lemma now follows from the bounds of Lemma~\ref{lem:I_rho_X_growth}.
\end{proof}

\subsection{Restricting the range of the partial sum estimate}\label{subsec:restrict_partial_sum} In \eqref{eq:CZ_bound} we assumed a
bound of the shape 
  \begin{equation}\label{eq:CZ_bound_2}
   \sum_{\mu_n \leq Z} \lvert b(n) \rvert \ll_\psi Z^r \log^{r'}(C_\psi' Z)
\end{equation}
for all $Z$ simultaneously. Here, as a refinement of our main theorem, we argue that this is only required when $Z$ is `approximately' bounded by the parameter $z$ of 
\eqref{eq:def_yz}.

More specifically, suppose for some $C_1 > 0$ that \eqref{eq:CZ_bound} holds simultaneously for all $Z \leq zX^{C_1}$, and
for $Z > z X^{C_1}$ assume only a (very weak) bound of the shape
  \begin{equation}\label{eq:CZ_bound_2}
   \sum_{\mu_n \leq Z} \lvert b(n) \rvert \leq C_\psi Z^{C_2}
\end{equation}
for {\itshape any} constant $C_2$. Then, Theorem \ref{thm:cn-full-improved} still holds, with the implied constant in 
\eqref{eq:thm_main_improved_positive} now depending additionally on $C_1$ and $C_2$. 

The proof is immediate: in \eqref{eq:W_split}, break the sum over $\mu_n > z$ into the ranges $z < \mu_n \leq zX^{C_1}$
and $\mu_n > zX^{C_1}$. The smaller range is estimated as before; for the larger range, choose the parameter $k$ large enough
(depending on $C_1$ and $C_2$) so that the bound \eqref{eq:CZ_bound_2} is enough to guarantee that the contribution is bounded
above by that of the smaller range.

We refer to \cite{BTT} for an application where this additional flexibility is required.

\section{A simpler version: Proof of Theorem \ref{thm:cn}}\label{sec:simpler}
For the reader's convenience, we give the (brief!) explanation of how
Theorem \ref{thm:cn} follows immediately from
Theorem \ref{thm:cn-full-improved}. Other variations can be proved in the same way.

We assumed that $\phi(s)$ has a `well behaved' functional equation. To make this precise, consider the following special case of
the conditions described in Section \ref{sec:notation}:
Assume that $\delta = 1$, so that the functional equation relates $s$ to $1 - s$. We assume that 
each $\alpha_v$ in \eqref{eq:def_Delta} equals $\frac{1}{2}$, so that
$d = N = 2A$ is the usual degree of the zeta function.
 We also assume that
both $\phi$ and $\psi$ are holomorphic away from
simple poles at $s = 1$. If $\psi$ has nonnegative coefficients, then this implies that there exists a positive constant $\widehat{\delta_1}$ for which 
we may take $B_\psi(Z) = \widehat{\delta_1} Z$ in \eqref{eq:supposition}; in any case, we assume that such a $\widehat{\delta_1}$ exists.

By definition, we have
\begin{equation}\label{eq:Rphi_ex}
R_\phi(X) = X \cdot \Res_{s = 1} \phi(s) + \phi(0).
\end{equation}
By the functional equation we have $\phi(0) \ll |\Res_{s = 1} \psi(s)|$, and
\begin{equation}
|\Res_{s = 1} \psi(s)| \leq \limsup_{s \rightarrow 1^+} \ (s - 1)\sum_{\mu_n} |b(n)| \mu_n^{-s}
\leq  \limsup_{s \rightarrow 1^+} \ (s - 1) \sum_{Z} |b(n)| Z^{1-s},
\end{equation}
with the last sum over all dyadic intervals $[Z, 2Z]$ on which the $\mu_n$ are supported. Writing $Z_{\min}$ for the smallest value of $\mu_n$, 
this last quantity is bounded by
\[
\widehat{\delta_1} \limsup_{s \rightarrow 1^+} \ (s - 1) Z_{\min}^{1 - s} \frac{1}{1 - 2^{1 - s}} \ll \widehat{\delta_1}.
\] 
Applying Theorem \ref{thm:cn-full-improved}, we thus obtain
\begin{equation}\label{eq:cn2}
\sum_{\lambda_n < X} a(n) - \Res_{s = 1} \big( \phi(s) \big) X
\ll
\delta_1 X^{1 - \frac{1}{d} - \eta} + \widehat{\delta_1}
X^{\frac{1}{2} - \frac{1}{2d}} \cdot (X^{d \eta})^{\frac{1}{2} - \frac{1}{2d}}.
\end{equation}

We equalize error terms by choosing $\eta$ so that $\delta_1 X^{\frac{1}{2} - \frac{1}{2d}}
= \widehat{\delta_1} X^{\eta} (X^{d \eta})^{\frac{1}{2} - \frac{1}{2d}}$, so that the error
is equal to $O(X^{ \frac{d - 1}{d + 1}} \delta_1^{ \frac{d - 1}{d + 1}} \widehat{\delta_1}^{\frac{2}{d +1 }})$,
as claimed in Theorem \ref{thm:cn}; the condition $\eta \geq - \frac{1}{2A}$ is equivalent to our demand that the error term
be bounded by the main term.

\begin{remark}
We also have the following {\itshape averaged} version of Theorem \ref{thm:cn}.
Suppose that $\big( \phi_i \big)_{i = 1}^n$ is a family of zeta functions, with functional equations
\[
\Delta(s) \phi_i(s) = \Delta(\delta - s) \psi_i(\delta - s)
\]
satisfying all of the hypotheses above for the same function $\Delta$. Then, we have
\begin{equation}\label{eq:cn_averaged}
\sum_{i = 1}^n
\left| \sum_{\lambda_{n, i} < X} a_i(n) - \Res_{s = 1} \big( \phi_i(s)\big) X \right|
\ll 
X^{\frac{d - 1}{d + 1}}
 \left( \sum_{i = 1}^n
\delta_{1, i}\right)^{\frac{d - 1}{d + 1}} 
\cdot 
\left(  \sum_{i = 1}^n \widehat{\delta_{1, i}} \right)^{\frac{2}{d + 1}}
\end{equation}
if the right hand side is bounded by the main term
$\sum_{i = 1}^n \Res_{s = 1} \big( \phi_i(s)\big)X$.
(In the above, the notation $a_i(n)$, $\lambda_{n, i}$, $\delta_{1, i}$, $\widehat{\delta_{1, i}}$ refers
to the quantities $a(n)$, $\lambda_n$, $\delta_i$, and $\widehat{\delta_i}$ associated to each $\phi_i$.)
The proof is immediate: in \eqref{eq:cn2}, choose a single $\eta$ to equalize the cumulative error terms, rather
than choosing an $\eta_i$ for each $\phi_i$. 

Although \eqref{eq:cn_averaged} follows immediately from H\"older's inequality and Theorem \ref{thm:cn}, the above
proof establishes that it is enough to assume that the error term in \eqref{eq:cn_averaged}
is bounded by the main term on average, as opposed to
individually for each $\phi_i$.
\end{remark}

\section{Ideals in number fields: Proof of Theorem \ref{thm:ideals}}\label{sec:proof_ideals}

The proof follows immediately from Theorem  \ref{thm:cn-full-improved} upon recalling the properties
of the associated Dedekind zeta function.
Recall (e.g. from \cite[Chapter 5.10]{IK}) that if $K/\Q$ is a number field of degree $d$,
then its {\itshape Dedekind zeta function}
\begin{equation}
\zeta_K(s) = \sum_{\mathfrak{a} \neq 0} (N \mathfrak{a})^{-s}
\end{equation}
satisfies the functional equation
\begin{equation}\label{eqn:dz_fe}
\Delta(s) \zeta_K(s) = \Delta(1 - s) \widetilde{\zeta_K}(1 - s),
\end{equation}
with
\begin{equation}\label{eq:def_Delta}
\Delta(s) = \Gamma \Big( \frac{s}{2} \Big)^{r_1 + r_2} \Big( \frac{s + 1}{2} \Big)^{r_2},
\end{equation}
where $r_1$ is the number of real embeddings of $K$ and $r_2$ the number of pairs
of complex conjugate embeddings (so that $d = r_1 + 2r_2$), $q := |\Disc(K)|$,
and
\begin{equation}
  \widetilde{\zeta_K}(s) = q^{s - \frac{1}{2}} \pi^{\frac{d}{2} - ds} \zeta_K(s)
  =
  \sum_{\mathfrak{a} \neq 0} q^{-\frac{1}{2}} \pi^{\frac{d}{2}} \Big( N \mathfrak{a} \cdot \frac{\pi^d}{q} \Big)^{-s}.
\end{equation}
The zeta function $\zeta_K(s)$ is entire, away from a simple pole at $s = 1$ with residue
\begin{equation}\label{eqn:dz_residue}
\Res_{s = 1} \zeta_K(s) = \frac{ 2^{r_1} (2 \pi)^{r_2} h R}{ w \sqrt{q} } \ll_d (\log q)^{d - 1}
\end{equation}
where $w$ is the number of roots of unity in $K$,
$h$ is the class number of $K$, $R$ is the regulator of $K$,
and where the upper bound is \cite[Theorem 1]{louboutin}.

We have $\zeta_K(0) \ll q^{1/2} (\log q)^{d - 1}$ by \eqref{eqn:dz_fe} and \eqref{eqn:dz_residue} (and indeed $\zeta_K(0) = 0$ if
$K$ is not imaginary quadratic), and we apply Theorem \ref{thm:cn-full-improved} with
\[
  \delta = 1, \ A = \frac{d}{2}, \ R_\phi(X) = X (\log q)^{d - 1} + O\big(q^{1/2} (\log q)^{d - 1}\big).
\]
We have $\zeta_K(s) \leq \zeta(s)^d = \sum_n d_d(n) n^{-s}$ coefficientwise, and
\begin{equation}\label{eq:dedekind_trivial}
\sum_{n < Z} d_d(n) \ll_d Z (\log Z)^{d - 1},
\end{equation}
so that we may take
\[
B_\psi(Z) = Z q^{1/2} (\log (Zq))^{d - 1}
\]
to conclude that
\begin{multline*}
\# \{ \mathfrak{a} \ : \ N(\mathfrak{a}) < X \}
- X \Res_{s = 1} \zeta_K(s) - O\big(q^{1/2} (\log q)^{d - 1}\big)
\ll \\
X^{1 - \frac{1}{d} - \eta} (\log q)^{d - 1}
+
q^{1/2} X^{\frac{1}{2} - \frac{1}{2d}} \cdot (X^{d \eta})^{\frac{1}{2} - \frac{1}{2d}}
(\log q X^{d \eta})^{d - 1}.
\end{multline*}
We choose $X^{\eta} = X^{\frac{d - 1}{d(d + 1)}} q^{-\frac{1}{d + 1}}$;
formally, this is equivalent to applying
Theorem \ref{thm:cn} with $\delta_1 \ll (\log q)^{d - 1}$ and
$\widehat{\delta_1} \ll q^{1/2} (\log qX)^{d - 1}$. (We may not literally apply
Theorem \ref{thm:cn} as stated because this $\widehat{\delta_1}$ depends on
$X$.) We also note that $\log(qX^{d \eta}) \ll_d \log(X)$ whenever $q \leq X$
(and if $q > X$, our conclusion does not beat the trivial bound \eqref{eq:dedekind_trivial}).

Putting everything together, we have
\[
\# \{ \mathfrak{a} \ : \ N(\mathfrak{a}) < X \} =
\frac{ 2^{r_1} (2 \pi)^{r_2} h R}{ w |\Disc(K)|^{1/2} } X
+
O\Big( |\Disc(K)|^{\frac{1}{d + 1}} X^{\frac{d - 1}{d + 1}} (\log X)^{d - 1} \Big).
\]

\section{Counting lattice points: Proof of Theorem~\ref{thm:main}}\label{sec:proof_main}

\subsection{Background on Epstein zeta functions}

We assemble some background material on Epstein zeta functions which will be
needed in the proof. Epstein's original paper is~\cite{epstein};
our formulation of his results can be found (for example) in ~\cite{Borwein2014}, but to our knowledge the only
reference for the proofs is Epstein's original work. We also refer to~\cite{Cassels}
for a good reference on lattices
and the geometry of numbers.

\medskip

If $\Lambda \subseteq \R^d$ is a rank $d$ lattice, then we choose a matrix $L \in \GL_d(\R)$ for which
$\Lambda = \{ Lx \ : \ x \in \Z^d \}$, and define
$\det \Lambda = \lvert \det L \rvert$.
($L$ is not uniquely defined, but $\det \Lambda$, $\Lambda^*$, and $\zeta(s, \Lambda)$ will be.)

We define the \emph{dual lattice} $\Lambda^*$ to be the set of all vectors $u
\in \mathbb{R}^d$ such that $u^T v \in \mathbb{Z}$ for every $v \in \Lambda$.
It is easy to show that $\Lambda^*$ is actually a lattice of rank $d$, and in
fact it is given by
\begin{equation*}
  \Lambda^* = \{ (L^T)^{-1} x: x \in \mathbb{Z}^d\}.
\end{equation*}
Thus $\Lambda$ is also the dual lattice of $\Lambda^*$, and $\det \Lambda \det
\Lambda^* = 1$.

The function $v \mapsto \lvert v \rvert^2$ is a positive definite quadratic form
on $\Lambda$: if $v = Lx$ where $x\in\Z^d$,
then $\lvert v \rvert^2 = Lx \cdot Lx = x^T (L^T L) x$.
Writing $Q = L^TL$ for the matrix associated to this quadratic form, we have
$\lvert v \rvert^2 = Q[x] := x^T Q x$ and $\det Q = \det (L^T L) = (\det \Lambda)^2$.

Then the \emph{Epstein zeta function} associated to $\Lambda$ (or to $Q$) is
defined by the Dirichlet series
\begin{equation}
  \zeta(s, \Lambda) := \zeta(s, Q) :=
\sum_{v\in\Lambda-\{0\}}|v|^{-2s}=
\sum_{x \in \Z^d - \{ 0 \}} Q[x]^{-s}.
\end{equation}
It converges absolutely for $\Re(s) > \frac{d}{2}$, has analytic continuation to
$\C$ apart from a simple pole at
$s = \frac{d}{2}$ with residue
\begin{equation*}
  \Res_{s = \frac{d}{2}} \zeta(s, \Lambda)
  =
  \frac{1}{\sqrt{\lvert \det Q \rvert}} \frac{\pi^{d/2}}{\Gamma(d/2)}
  =
  \frac{1}{\lvert \det \Lambda \rvert} \frac{\pi^{d/2}}{\Gamma(d/2)},
\end{equation*}
and satisfies the functional equation
\begin{equation}
  \pi^{-s} \Gamma(s) \zeta(s, \Lambda)
  =
  (\det \Lambda)^{-1} \pi^{s - \frac{d}{2}} \Gamma( \tfrac{d}{2} - s )
  \zeta( \tfrac{d}{2} - s, \Lambda^* ).
\end{equation}

\subsection{Conclusion of the proof}

In the introduction, we noted that we can prove Theorem~\ref{thm:main}
by rescaling both $\zeta(s, \Lambda)$ and the output of Theorem~\ref{thm:cn}.
But by using Theorem~\ref{thm:cn-full-improved}, it is possible to avoid any scaling.

We apply Theorem~\ref{thm:cn-full-improved} with $X = R^2$, $\phi(s) = \zeta(s, \Lambda)$,
$\psi(s) = (\det \Lambda)^{-1}  \pi^{\frac{d}{2} -2s} \zeta(s, \Lambda^*)$, $\delta =
\frac{d}{2}, A = 1$. We obtain
\begin{equation}\label{eqn:main2}
  N(\Lambda, R)
  =
  \frac{\pi^{d/2}}{\Gamma(d/2 + 1)} \frac{R^d}{\lvert \det(\Lambda) \rvert}
  +
  O_{d}\bigg(\frac{1}{|\det \Lambda|} R^{d - 1 - 2 \eta}
  + \frac{1}{|\det \Lambda|}    C'_{\zeta(\cdot, \Lambda^*)}
 R^{(1 + 2\eta)\big(\frac{d}{2} - \frac{1}{2}\big)} \bigg),
 \end{equation}
where $C'_{\zeta(\cdot, \Lambda^*)}$ is a positive constant guaranteeing the
bound
\begin{equation}\label{eq:dual_point_count}
  \sum_{\substack{v \in \Lambda^* \\ 0 <  |v|^2 \leq Z}} 1 \leq
  C'_{\zeta(\cdot, \Lambda^*)} Z^{d/2},
\end{equation}
and upon choosing $R^{2 \eta} = R^{\frac{d - 1}{d + 1}} (C'_{\zeta(\cdot, \Lambda^*)})^{- \frac{2}{d + 1}}$
we obtain
\begin{equation}\label{eqn:main3}
  N(\Lambda, R)
  =
  \frac{\pi^{d/2}}{\Gamma(d/2 + 1)} \frac{R^d}{\lvert \det(\Lambda) \rvert}
  +
  O_{d}\bigg(
      \frac{1}{\lvert \det \Lambda \rvert}
    (C'_{\zeta(\cdot, \Lambda^*)})^{\frac{2}{d+1}}
    \; R^{d \cdot \frac{d-1}{d+1}}
  \bigg),
\end{equation}

We will see that the condition $\eta \geq - \frac{1}{2A}$ holds whenever $R > r_{\bas}$; 
it remains to bound $C'_{\zeta(\cdot, \Lambda^*)}$, which we
do in the next lemma.

\begin{lemma}\label{lem:point_count}
  For any complete lattice $\Lambda \subseteq \R^d$, let $\lambda_1(\Lambda)$ denote the length of the shortest nontrivial vector in
  $\Lambda$.
  The number of lattice points in $\Lambda$ satisfying $|v| \leq X$ is
  bounded by
  \begin{equation}\label{eq:lattice_1}
    \sum_{\substack{v \in \Lambda \\ 0 < |v| \leq X}} 1
    \ll_d
    \frac{X^d}{\lambda_1(\Lambda)^d}.
  \end{equation}
  Therefore, in the notation above, $C'_{\zeta(\cdot, \Lambda)}$ can be taken as
  \begin{equation*}
    C'_{\zeta(\cdot, \Lambda)} = \frac{c_d}{\lambda_1(\Lambda)^d}
  \end{equation*}
  for some absolute constant $c_d$ depending only on the dimension $d$.
\end{lemma}

\begin{proof}
Assume $X \leq \lambda_1(\Lambda)$ (otherwise, the bound is trivial), and
define $R_j := \{ x \in \mathbb{R}^d : j\lambda_1(\Lambda) \leq |x| < (j+1)\lambda_1(\Lambda)\}$
to be a set of $d$-dimensional annuli, so that
\begin{equation*}
  \sum_{\substack{v \in \Lambda \\ |v| \leq X}} 1
  =
  \sum_{j \geq 1} \;
  \sum_{\substack{v \in \Lambda \cap R_j \\  |v| \leq X}} 1  \leq
    \sum_{j \leq \lfloor X/\lambda_1(\Lambda)\rfloor} \# \{ v \in \Lambda \cap R_j \} .
\end{equation*}
To bound $  \# \{ v \in \Lambda \cap R_j \}$, consider $n$-spheres of radius
$\frac{\lambda_1(\Lambda)}{2}$ around each $v$ being counted: their interiors
are disjoint and lie within the annulus
$\{ \frac{|x|}{\lambda_1(\Lambda)} \in [j - \frac12, j + \frac32] \}$,
so that
\begin{equation*}
  \# \{ v \in \Lambda \cap R_j \}
  \ll_d
  \Big( j + \frac32 \Big)^d -
  \Big( j - \frac12 \Big)^d \ll_d j^{d - 1},
\end{equation*}
yielding the bound
\begin{equation}
  \sum_{\substack{v \in \Lambda \\  |v|  \leq X}} 1
  \ll_d
  \sum_{j \leq \lfloor X/\lambda_1(\Lambda)\rfloor} j^{d-1}
  \ll_d
  \frac{X^d}{\lambda_1(\Lambda)^d},
\end{equation}
where the implicit constants depend only on the dimension $d$,
and not on $\Lambda$.
\end{proof}

\begin{remark} Observe that, owing to the shape of the functional
equation of the Epstein zeta function, the proof of Theorem \ref{thm:main} requires as input 
a simpler but similar statement.

One can improve the bound in Theorem \ref{thm:main}, at the expense of complicating its statement,
by incorporating a stronger bound than Lemma \ref{lem:point_count}. For example, by Widmer's
bound \cite[Theorem 5.4]{widmer}, we have
  \begin{equation}\label{eq:lattice_2}
    \sum_{\substack{v \in \Lambda \\ |v| \leq X}} 1
    \ll_d 1 + 
    \max_{1 \leq k \leq d} \frac{X^k}{\lambda_1(\Lambda) \cdots \lambda_k(\Lambda)}.  
    \end{equation}

 \end{remark}

\begin{lemma}\label{lem:dual_lattice}
  Suppose $\Lambda$ is any rank $d$ lattice in $\mathbb{R}^d$, and let
  $\Lambda^*$ denote its dual lattice.
  Let $r_{\bas}(\Lambda^*)$ denote the infimum of all $r \in \mathbb{R}^+$ such that the
  ball $B(r)$ contains a basis for $\Lambda^*$.
  Then
  \begin{equation*}
    \lambda_1(\Lambda) \cdot r_{\bas}(\Lambda^*) \geq 1.
  \end{equation*}
\end{lemma}

\begin{proof}
  Recall the definition of the dual lattice,
  $\Lambda^* :=
  \{ w \in \mathbb{R}^d
    :
    \forall v \in \Lambda, \langle v, w \rangle \in \mathbb{Z}
  \}$.
  Let $v \in \Lambda$ be of minimal length, so that $\| v \| =
  \lambda_1(\Lambda)$.
  Suppose $w_1, \ldots, w_d$ is a set of $n$ linearly independent elements in
  $\Lambda^*$ fitting within $\overline{B}(V_{\Lambda^*})$.
  Then there exists $i$ such that $\langle w_i, v \rangle \neq 0$.
  Then by the definition of $\Lambda^*$ above, we have $\langle w_i, v \rangle
  \in \mathbb{Z}$, and thus $ | w_i | | v | \geq 1$.
\end{proof}

By Lemmas~\ref{lem:point_count} and~\ref{lem:dual_lattice}, we can take
\begin{equation}\label{eq:foo}
  C'_{\zeta(\cdot, \Lambda^*)}
  =
  c_d/\lambda_1(\Lambda^*)^d
  \leq
  c_d \; r_{\bas}(\Lambda)^d.
\end{equation}
The condition $\eta \geq - \frac{1}{2A}$ is equivalent to $R^{\frac{d - 1}{d + 1}} \big(C'_{\zeta(\cdot, \Lambda^*)}\big)^{- \frac{2}{d + 1}}
\geq R^{-1}$, which by \eqref{eq:foo} is true if $r_{\bas}(\Lambda) \ll_d R$. We may then allow $r_{\bas} < R$ by multiplying $R^{\eta}$ by a
factor that is $O_d(1)$, which multiplies the error term in \eqref{eqn:main3} by another (harmless) factor of $O_d(1)$.

Our result therefore follows by inserting~\eqref{eq:foo} into~\eqref{eqn:main3}.

\section*{Acknowledgments}

We would like to thank Bruce Berndt, Eric Gaudron, Aleksandar Ivi\'c, Jesse Kass,
Martin Nowak, Anders S\"odergren, Jesse Thorner, and Martin Widmer for their advice, suggestions, and encouragement.

We began this paper at the Mathematical Sciences Research Insitute in Berkeley, CA in Spring 2017,
and we would like to thank MSRI for providing an excellent atmosphere in which to work, as well as the National Science Foundation (under Grant No. DMS-1440140) for their
financial support of MSRI.

The first author was partially supported by the National Science Foundation Graduate Research
Fellowship Program under Grant No.\ DGE 0228243 and the EPSRC Programme Grant
EP/K034383/1 LMF: L-Functions and Modular Forms;
the second author was partially supported by the JSPS, KAKENHI Grant Numbers
JP24654005, JP25707002, and JP17H02835;
the third author was partially supported by the National Security Agency under
Grant No. H98230-16-1-0051.

\bibliographystyle{alpha}
\bibliography{lattice-points}

\begin{thebibliography}{IKKN06}

\bibitem[Ang14]{Ange}
Thomas Ange.
\newblock Le th\'eor\`eme de {S}chanuel dans les fibr\'es ad\'eliques
  hermitiens.
\newblock {\em Manuscripta Math.}, 144(3-4):565--608, 2014.

\bibitem[BBS14]{Borwein2014}
David Borwein, Jonathan~M. Borwein, and Armin Straub.
\newblock On lattice sums and {W}igner limits.
\newblock {\em J. Math. Anal. Appl.}, 414(2):489--513, 2014.

\bibitem[BG97]{bentkus}
V.~Bentkus and F.~G\"otze.
\newblock On the lattice point problem for ellipsoids.
\newblock {\em Acta Arith.}, 80(2):101--125, 1997.

\bibitem[BTT]{BTT}
M.~Bhargava, T.. Taniguchi, and F.~Thorne.
\newblock Improved error estimates for the {D}avenport-{H}eilbronn theorems.
\newblock {\em In preparation}.

\bibitem[Cas97]{Cassels}
J.~W.~S. Cassels.
\newblock {\em An introduction to the geometry of numbers}.
\newblock Classics in Mathematics. Springer-Verlag, Berlin, 1997.
\newblock Corrected reprint of the 1971 edition.

\bibitem[CN62]{ChandrasekharanNarasimhan62}
K.~Chandrasekharan and Raghavan Narasimhan.
\newblock Functional equations with multiple gamma factors and the average
  order of arithmetical functions.
\newblock {\em Ann. of Math. (2)}, 76:93--136, 1962.

\bibitem[Dav51]{davenport_lemma}
H.~Davenport.
\newblock On a principle of {L}ipschitz.
\newblock {\em J. London Math. Soc.}, 26:179--183, 1951.

\bibitem[Deb16]{Debaene}
Korneel Debaene.
\newblock Explicit counting of ideals and a {B}run-{T}itchmarsh inequality for
  the {C}hebotarev {D}ensity {T}heorem.
\newblock {\em Preprint}, 2016.
\newblock Available at \url{https://arxiv.org/abs/1611.10103}.

\bibitem[Eps03]{epstein}
P.~Epstein.
\newblock Zur theorie allgemeiner zetafunctionen.
\newblock {\em Math. Ann.}, 56:615--644, 1903.

\bibitem[FI05]{FI}
John~B. Friedlander and Henryk Iwaniec.
\newblock Summation formulae for coefficients of {$L$}-functions.
\newblock {\em Canad. J. Math.}, 57(3):494--505, 2005.

\bibitem[IK04]{IK}
Henryk Iwaniec and Emmanuel Kowalski.
\newblock {\em Analytic number theory}, volume~53 of {\em American Mathematical
  Society Colloquium Publications}.
\newblock American Mathematical Society, Providence, RI, 2004.

\bibitem[IKKN06]{lattice_survey}
A.~Ivi\'c, E.~Kr\"atzel, M.~K\"uhleitner, and W.~G. Nowak.
\newblock Lattice points in large regions and related arithmetic functions:
  recent developments in a very classic topic.
\newblock In {\em Elementare und analytische {Z}ahlentheorie}, volume~20 of
  {\em Schr. Wiss. Ges. Johann Wolfgang Goethe Univ. Frankfurt am Main}, pages
  89--128. Franz Steiner Verlag Stuttgart, Stuttgart, 2006.

\bibitem[KT]{KT}
J.~Kass and F.~Thorne.
\newblock What is the height of two points in the plane?
\newblock {\em In preparation}.

\bibitem[Lan12]{Landau1912}
Edmund Landau.
\newblock {\"U}ber die {A}nzahl der {G}itterpunkte in gewissen {B}ereichen.
\newblock {\em Nachrichten von der Gessellschaft der Wissenschaften zu
  G\"ottingen, Mathematisch-Physikalische Klasse}, pages 687--770, 1912.

\bibitem[Lan15]{Landau1915}
Edmund Landau.
\newblock {\"U}ber die {A}nzahl der {G}itterpunkte in gewissen {B}ereichen.
  {Z}weite abhandlung.
\newblock {\em Nachrichten von der Gessellschaft der Wissenschaften zu
  G\"ottingen, Mathematisch-Physikalische Klasse}, pages 209--243, 1915.

\bibitem[LD17]{LowryDudaThesis}
David Lowry-Duda.
\newblock {\em On Some Variants of the {G}auss Circle Problem}.
\newblock PhD thesis, Brown University, 5 2017.
\newblock \url{https://arxiv.org/abs/1704.02376}.

\bibitem[Lou01]{louboutin}
St\'ephane Louboutin.
\newblock Explicit upper bounds for residues of {D}edekind zeta functions and
  values of {$L$}-functions at {$s=1$}, and explicit lower bounds for relative
  class numbers of {CM}-fields.
\newblock {\em Canad. J. Math.}, 53(6):1194--1222, 2001.

\bibitem[SS74]{SS}
Mikio Sato and Takuro Shintani.
\newblock On zeta functions associated with prehomogeneous vector spaces.
\newblock {\em Ann. of Math. (2)}, 100:131--170, 1974.

\bibitem[TT13]{TT_rc}
Takashi Taniguchi and Frank Thorne.
\newblock Secondary terms in counting functions for cubic fields.
\newblock {\em Duke Math. J.}, 162(13):2451--2508, 2013.

\bibitem[Wat95]{WatsonBessel}
G.~N. Watson.
\newblock {\em A treatise on the theory of {B}essel functions}.
\newblock Cambridge Mathematical Library. Cambridge University Press,
  Cambridge, 1995.
\newblock Reprint of the second (1944) edition.

\bibitem[Wid10]{widmer}
Martin Widmer.
\newblock Counting primitive points of bounded height.
\newblock {\em Trans. Amer. Math. Soc.}, 362(9):4793--4829, 2010.

\end{thebibliography}

 \end{document}